\documentclass[anon,12pt]{article}

\usepackage{algorithm,algorithmic}
\usepackage{bm}
\usepackage{tikz}
\usepackage{mathtools}
\usepackage{hyperref}
\usepackage{float}
\usepackage{amsmath}
\usepackage{amsthm}
\usepackage{amsbsy}
\usepackage{amsfonts}
\usepackage{amssymb}
\usepackage{fullpage}
\usepackage{times}

\usetikzlibrary{positioning,arrows}
\newcommand{\eq}[1]{Eq.~(\ref{eq:#1})}

\DeclarePairedDelimiterX{\norm}[1]{\lVert}{\rVert}{#1}
\DeclarePairedDelimiterX{\normtto}[1]{\lVert}{\rVert_{2\rightarrow 1}}{#1}
\DeclarePairedDelimiterX{\abs}[1]{\lvert}{\rvert}{#1}
\DeclarePairedDelimiterX{\inp}[2]{\langle}{\rangle}{#1, #2}

\DeclareMathOperator{\Tr}{Tr}
\DeclareMathOperator*{\argmax}{arg\,max}

\newcommand{\lprp}[1]{\left(#1\right)}
\newcommand{\lbrb}[1]{\left\{#1\right\}}
\newcommand{\lsrs}[1]{\left[#1\right]}
\newcommand{\mb}[1]{\mathbb{#1}}

\newcommand{\mt}{MT}
\newcommand{\rad}{300\lprp{\sqrt{\Tr \Sigma/n} + \sqrt{k\norm{\Sigma} / n}}}
\newcommand{\rady}{300\lprp{\sqrt{\Tr \Lambda/k} + \sqrt{\norm{\Lambda}}}}

\newcommand{\radAlg}{6000\lprp{\sqrt{\Tr \Sigma/n} + \sqrt{k\norm{\Sigma} / n}}}
\newcommand{\radAlgG}{7500\lprp{\sqrt{\Tr \Sigma/n} + \sqrt{k\norm{\Sigma} / n}}}
\newcommand{\radRet}{8000\lprp{\sqrt{\Tr \Sigma / n} + \sqrt{k \norm{\Sigma} / n}}}
\newcommand{\finret}{480000 \lprp{\sqrt{\frac{\Tr \Sigma}{n}} + \sqrt{\frac{\norm{\Sigma}\log 1 / \delta}{n}}}}
\newcommand{\nbu}{3200 \log 1 / \delta}
\newcommand{\nit}{1000 \log \norm{\mu} / \epsilon}
\newcommand{\proj}{\mathcal{P}}
\newcommand{\projp}{\mathcal{P}^\perp}
\newcommand{\optx}{x^*}
\newcommand{\optd}{d^*}
\newcommand{\xt}[1]{x_{#1}}
\newcommand{\dt}[1]{d_{#1}}
\newcommand{\gt}[1]{g_{#1}}
\newcommand{\dd}{\Delta}
\newcommand{\mte}{MTE}
\newcommand{\nite}{50 \log \norm{\mu} / \epsilon}
\newcommand{\rade}{300\lprp{\sqrt{\Tr \Sigma/n} + \sqrt{\norm{\Sigma}k/n}}}
\newcommand{\radae}{1200\lprp{\sqrt{\Tr \Sigma/n} + \sqrt{\norm{\Sigma}k/n}}}
\newcommand{\gpdbe}{1500\lprp{\sqrt{\Tr \Sigma/n} + \sqrt{\norm{\Sigma}k/n}}}
\newcommand{\finrete}{108000 \lprp{\sqrt{\frac{\Tr \Sigma}{n}} + \sqrt{\frac{\norm{\Sigma}\log 1 / \delta}{n}}}}

\newcommand{\vt}[1]{g_{#1}}

\newtheorem{assumption}{Assumption}
\newtheorem{definition}{Definition}
\newtheorem{theorem}{Theorem}
\newtheorem{lemma}{Lemma}
\newtheorem{corollary}{Corollary}
\newtheorem{remark}{Remark}

\title{Fast Mean Estimation with Sub-Gaussian Rates}
\usepackage{times}
\author{%
 Yeshwanth Cherapanamjeri\\
 \texttt{yeshwanth@berkeley.edu}\\
 UC Berkeley
 \and
 Nicolas Flammarion\\
 \texttt{flammarion@berkeley.edu}\\
 UC Berkeley
 \and
 Peter L. Bartlett\\
 \texttt{peter@berkeley.edu}\\
 UC Berkeley
}
\date{}

\begin{document}

\maketitle

\begin{abstract}
  We propose an estimator for the mean of a random vector in $\mb{R}^d$ that can be computed in time $O(n^4+n^2d)$ for $n$ i.i.d.~samples and that has error bounds matching the sub-Gaussian case. The only assumptions we make about the data distribution are that it has finite mean and covariance; in particular, we make no assumptions about higher-order moments. Like the polynomial time estimator introduced by \cite{hopkins2018sub}, which is based on the sum-of-squares hierarchy, our estimator achieves optimal statistical efficiency in this challenging setting, but it has a significantly faster runtime and a simpler analysis.
\end{abstract}

\section{Introduction}

Estimating the mean of a population given a finite sample is arguably the most fundamental statistical estimation problem.
Despite the broad applicability and the fundamental nature of this problem, an estimator achieving the optimal statistical rate has only been discovered recently. However the optimal computational complexity of such an estimator is not well-understood.

In this paper, we are interested in obtaining high confidence estimates of the mean in the simple setting where only the existence of  the covariance of the distribution is assumed. 
That is, we would like to find  the smallest $r_\delta$ such that given samples $X_1, \dots, X_n$ from a distribution $\mathcal{D}$ with mean $\mu$   our estimator $\hat{X}$  satisfies:
\begin{equation*}
  \mb{P} \lbrb{\norm{\hat{X} - \mu} \geq r_\delta} \leq \delta.
\end{equation*}

To understand the inherent statistical limit of this problem, let us consider the simplified setting where the covariance is the identity. The most natural estimator for the mean of the population is the sample mean $\bar{X}=\frac{1}{n}\sum_{i=1}^nX_i$. From the Central Limit Theorem, the distribution of $\bar{X}$ satisfies  $\sqrt{n} (\bar{X}\!-\!\mu)\!\overset{D}{\rightarrow}\! \mathcal{N} (0, I)$, and assuming this conclusion holds for any $n$ allows an $r_\delta$~satisfying
\begin{equation*}
  r_\delta = O\lprp{\sqrt{\frac{d}{n}} + \sqrt{\frac{\log 1 / \delta}{n}}}.
\end{equation*}
\cite{catoni2012challenging} shows that this $r_\delta$ is the optimal statistical performance achievable under such mild assumptions. However, the above confidence interval only holds true asymptotically when the number of samples goes to infinity or when the distribution is sub-Gaussian. For finite sample results with a heavy-tailed distribution, applying Chebyshev's inequality to the empirical mean gives only
\begin{equation*}
  r_\delta = \Omega\lprp{\sqrt{\frac{d}{n\delta}}}.
\end{equation*}
The above bound is weaker than the one obtained by the Central Limit Theorem in two ways, the dependence on the failure probability $\delta$ is polynomial in $1 / \delta$ instead of logarithmic and the term depending on $\delta$ is multiplied by the dimensionality $d$ as opposed to being part of a smaller additive term. Unfortunately, \cite{catoni2012challenging} also shows the above result is tight. That is, for any $n, \delta$, there exists a distribution $\mathcal{D}_{n, \delta}$ for which the bound guaranteed by Chebyshev's inequality is optimal.

The poor performance of the empirical mean is due to its sensitivity to large outliers that occur naturally as part of the sample. The median-of-means framework was devised as a means of circumventing such difficulties. It was independently developed in the one dimensional case by \cite{NemYud83,jerrum1986random,alon1999space} and was later extended to the multivariate case by~\cite{HsuSab16,lerasle2011robust,Min18}. As part of this framework, the samples are first divided into $k$ batches and the mean of the samples is computed within each batch to obtain $k$ estimates $Z_1, \dots, Z_k$. Each of these has mean $\mu$ and variance $\frac{k}{n}I$. The empirical mean is simply the mean of these $k$ estimates, which is sensitive to outliers. The median-of-means estimator 
instead is the geometric median of the $k$ estimates, which has greater tolerance to outliers. The success of the median-of-means estimator is due to the fact that it relies on only a fraction of estimates $Z_i$ being close to the mean as opposed to all the estimates being close. \cite{Min18} shows this gives an improved value of $r_\delta$ as follows:
\begin{equation*}
  r_\delta = O\lprp{\sqrt{\frac{d \log 1 / \delta}{n}}}.
\end{equation*}
The confidence interval guaranteed by the median-of-means estimator is better than the one for the empirical mean by improving the dependence on $1 / \delta$, but it is still poorer than we might expect from the Central Limit Theorem. Subsequent work attempting to bridge this gap achieves better rates than those guaranteed by the median-of-means but with stronger assumptions on the data generating distribution\footnote{A rate of $O\big(\sqrt{d / n} + \sqrt{\log (\frac{\log d}{\delta}) /  n}\big)$ is achieved under a fourth moment assumption on the distribution.}~(\cite{joly2017}).
The question of whether it was statistically feasible to obtain confidence intervals of the form guaranteed by the Central Limit Theorem was finally resolved by \cite{lugosi2017sub}. They devised an improved estimator, based on the median-of-means framework, called the median-of-means tournament, which achieves CLT-like confidence intervals. While the median-of-means estimator relies on the concentration of the number of $Z_i$ close to the mean in Euclidean norm, the median-of-means tournament relies on the fact that along every direction $v$, the number of $Z_i$ close to the projection of the mean concentrates. The freedom to choose a different set of $Z_i$ for each direction allow one to obtain a much smaller confidence interval than the one for the median-of-means estimator. In subsequent work, following the PAC-Bayesian approach of \cite{catoni2012challenging}, \cite{catoni2017dimension} proposed a soft-truncation based estimator which obtains CLT-like confidence intervals provided one has access to estimates of the trace and spectral norm of the covariance matrix.

However, it is not known whether the estimators from \cite{lugosi2017sub,catoni2017dimension} are computationally feasible, as there are no known polynomial time algorithms to compute them. In contrast, the median-of-means and empirical mean can be computed in nearly-linear time (\cite{cohen2016}). To alleviate this computational intractability, \cite{catoni2018dimension} proposed an efficient polynomial time estimator which achieves optimal statistical performance up to second order terms, assuming the existence of higher order moments. The question of computational tractability was subsequently resolved by \cite{hopkins2018sub}, who showed that an algorithm based on a sum-of-squares relaxation of the median-of-means tournament estimator achieves the statistically optimal CLT-like confidence intervals. However, the runtime of this algorithm is exorbitantly large\footnote{Assuming standard runtimes of the Interior Point method for semidefinite programming~(\cite{alizadeh1995interior})} ($O\lprp{n^{24}}$). 

In this paper, we propose a new algorithm with a reduced runtime---$O(n^4+n^2 d)$---and a significantly simpler analysis. Our algorithm is a descent-based method that iteratively improves an estimate of the mean. The main challenge of such an approach is to estimate the descent direction. To this end, we crucially leverage the structure of the solutions to semidefinite programming relaxations of polynomial optimization problems designed to test whether a estimate is close to the mean. Our main contributions are twofold; we first show how exact solutions to the polynomial optimization problem furnish suitable descent directions and that such descent directions can also be efficiently extracted from relaxations of such problems and secondly, we show that these descent directions can be used in a descent style algorithm for mean estimation.
Our paper is organized as follow: in Section~\ref{sec:mnRes}, we present our main result, then in Section~\ref{sec:intuition}, as a warm-up, we devise a descent style algorithm for the case where we are given exact solutions to the polynomial optimization problems mentioned previously and prove that this algorithm achieves optimal statistical efficiency. This sets the stage for Section~\ref{sec:efficient}, where we present our main algorithm based on semidefinite relaxations of the previously defined polynomial optimization problems, leading to computationally efficient sub-Gaussian mean estimation.

\section{Main result}
\label{sec:mnRes}
Formally, our main result\footnote{The constants are explicit but we believe sub-optimal.} is as follows:
\begin{theorem}
  \label{thm:sgmest}
  Let $\bm{X} = (X_1, \dots, X_n) \in \mb{R}^{n \times d}$ be $n$ i.i.d.~random vectors with mean $\mu$ and covariance $\Sigma$. Then Algorithm~\ref{alg:meste} instantiated with Algorithms~\ref{alg:dest} and~\ref{alg:gest} and run with inputs $\bm{X}$, target confidence $\delta$, stepsize $\gamma=1/20$ and number of iterations $T = \nit$ returns a vector $\optx$ satisfying:
  \begin{equation*}
    \norm{\optx - \mu} \leq \max\lprp{\epsilon, \finret},
  \end{equation*}
  with probability at least $1 - \delta$.
\end{theorem}
We can make the following comments:
\begin{itemize}
    \item Our estimator is both statistically optimal and computationally efficient. It achieves sub-Gaussian performance under minimal conditions on the distribution, and its runtime is  \break \mbox{$O(n^4+n^2d)$}. See Section~\ref{sec:algorelax} for details.
    \item The dependence of the number of iterations, $T$, on $\norm{\mu}$ can be avoided by initializing the algorithm with the median-of-means estimate. In this case, we can instead use $T = 1000 \log d$ and obtain the same guarantees, avoiding any dependence on the knowledge of $\norm{\mu}, \Tr(\Sigma), \norm{\Sigma}$.
    \item The estimator depends on the confidence level $\delta$. \cite{devroye2016sub} propose an estimator which works for a whole range of $\delta$ but for a restricted class of distributions.
    \item Our result does not explicitly depend on the dimension $d$ and our algorithm can be extended to a Hilbert space by working within the finite dimensional subspace containing the data points.
  \end{itemize}

\section{Warm-up}
\label{sec:intuition}
We present in this section a simple descent based algorithm. This algorithm is computationally inefficient but achieves the same guarantees of Theorem~\ref{thm:sgmest} with a much simpler analysis which nevertheless illustrates the main ideas behind the algorithm and proof of Theorem~\ref{thm:sgmest}.
\subsection{Intuition}
\label{sec:mte}
\begin{figure}[t]
  \centering
  \begin{tikzpicture}[scale=0.5]
    \node (x) at (-10, 0) [circle,fill,inner sep=1.5pt]{};
    \node [below=0.1cm of x] {$x$};
    \node (mu) at (5, 0) [circle,fill,inner sep=1.5pt]{};
    \node [below=0.1cm of mu] {$\mu$};
    \node at (2.6918,   3.7240) [blue,circle,fill,inner sep=1.5pt]{};
    \node at ( 0.8021,  2.4757) [blue,circle,fill,inner sep=1.5pt]{};
    \node at (8.7730,  4.1369) [blue,circle,fill,inner sep=1.5pt]{};
    \node at (-0.3608, -3.1745) [blue,circle,fill,inner sep=1.5pt]{};
    \node at ( 9.8048, -1.4058) [blue,circle,fill,inner sep=1.5pt]{};
    \node at ( 5.6202, -0.8174) [blue,circle,fill,inner sep=1.5pt]{};
    \node at (9.1835,  3.2953) [blue,circle,fill,inner sep=1.5pt]{};
    \node at (-4.8045, -0.8336) [blue,circle,fill,inner sep=1.5pt]{};
    \node at ( 4.0115,  2.1046) [blue,circle,fill,inner sep=1.5pt]{};
    \node at (1.2392,  -3.6) [blue,circle,fill,inner sep=1.5pt]{};
    \draw[-latex,thick] (-10,0) -- node[below] {$\hat{\delta}$} (5,0);
    \draw[thick] (-1,4.5) -- (1,-4);
    \draw[-latex,thick] (-10,0) -- node[above] {$v$} (-0.4,1.92);
    \draw[fill=red,opacity=0.3]  (-1,4.5) -- (10,4.5) -- (10,-4) -- (1,-4) -- cycle;
  \end{tikzpicture}
      \vspace{-.5em}
  \caption{ The direction $v$ solution to \mte{} is well aligned with the vector joining the current estimate $x$ to the true mean $\mu$.}
    \vspace{-1.5em}
  \label{fig:int}
\end{figure}
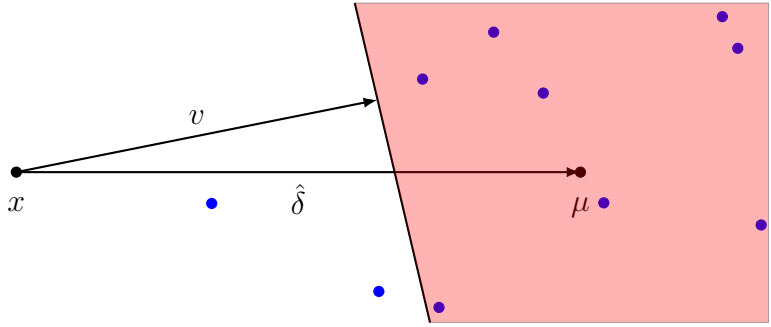
We provide some intuition for our procedure, which iteratively improves an estimate of the mean. We first consider the simpler problem of testing whether a given point is close to the mean. We draw our inspiration from the main technical insight of \cite{lugosi2017sub}, who show that along any direction, most of the bucket means, $Z_i$, are close to the mean, $\mu$. Thus, to test whether a point, $x$, is far from the mean, it is sufficient to check whether there exists a direction where most of the $Z_i$ are far away from $x$ along that direction. This is formally expressed in the following polynomial optimization problem:
      \vspace{-.7em}
\begin{gather*}
  \max \sum_{i = 1}^{k} b_i \\
  b_i^2 = b_i \\
  \norm{v}^2 = 1 \\
  b_i \inp{v}{Z_i - x} \geq b_i^2 r \quad \forall i \in [k] \label{eq:mte} \tag{\textbf{MTE}}
\end{gather*}
This polynomial problem over the set of variables $b_1, \dots, b_k$ and $v_1, \dots, v_d$ is parameterized by $r > 0$, the current estimate $x \in \mb{R}^d$ and the bucket means $\bm{Z} \in \mb{R}^{k\times d}$. Its polynomial constraints are encoding the number of $Z_i$ beyond a distance  $r$ from $x$ when projected along a direction $v$.
Intuitively,  this program tries to find a direction $v$ so as to maximize the number of $Z_i$ beyond a distance $r$ from $x$ along that direction. Here, we know from (\cite{lugosi2017sub}) that for an appropriate choice of $r$, along all directions $v$, a large fraction of the $Z_i$ are close to the mean. Formally, for all directions $v$, $\abs{\{i: \abs{\inp{Z_i - \mu}{v}} \leq r\}} \geq 0.9k$ (see Corollary~\ref{cor:dconce}
). Therefore this optimization problem has a large value  when $x$ is far from the mean and can be used to certify this. 

Strikingly, the direction $v$ returned by the solution of the above problem also contains information about the location of the mean when $r$ is chosen appropriately, which enables improvement of the quality of the current estimate. As illustrated in Figure~\ref{fig:int}, the direction returned by this optimization problem is strongly correlated with the vector joining the current point $x$ to the mean~$\mu$.
\begin{algorithm}[H]
  \caption{Mean Estimation}
  \label{alg:meste}
  \begin{algorithmic}[1]
    \STATE \textbf{Input}: Data Points $\bm{X} \in \mb{R}^{n\times d}$, Target Confidence $\delta$, Number of Iterations $T$, Stepsize $\gamma$

    \STATE $k \leftarrow \nbu$
    \STATE Split data points into $k$ bins with bin $\mathcal{B}_i$ consisting of the points $X_{(i - 1)\frac{n}{k} + 1}, \dots, X_{i\frac{n}{k}}$
    \STATE $Z_i \leftarrow \text{Mean}(\mathcal{B}_i)\ \forall\ i \in [k]$ and  $\bm{Z} \leftarrow (Z_1, \dots, Z_k)$
  \STATE $\optx, \xt{0} \leftarrow \bm{0}$ and  $\optd, \dt{0} \leftarrow \infty$

    \FOR{$t = 0:T$}
      \STATE $\dt{t} \leftarrow \text{Distance Estimation}(\bm{Z}, \xt{t})$
      \STATE $\gt{t} \leftarrow \text{Gradient Estimation}(\bm{Z}, \xt{t})$

      \IF {$\dt{t} < \optd$}
        \STATE $\optx \leftarrow \xt{t}$
        \STATE $\optd \leftarrow \dt{t}$
      \ENDIF

      \STATE $\xt{t + 1} \leftarrow \xt{t} + \gamma \dt{t}\gt{t}$
    \ENDFOR
    \STATE \textbf{Return: } $\optx$
  \end{algorithmic}
\end{algorithm}
\vspace{-.5cm}
\noindent
\begin{minipage}[H]{.49\textwidth}
\begin{algorithm}[H]
  \caption{Distance Estimation}
  \label{alg:deste}
  \begin{algorithmic}[1]
    \STATE \textbf{Input}: Data Points $\bm{Z} \in \mb{R}^{k\times d}$, Current point $x$
    \STATE $d^* = \argmax_{r > 0} \mte(x,r,\bm{Z}) \geq 0.9 k$
    \STATE \textbf{Return: } $d^*$\newline
  \end{algorithmic}
\end{algorithm}
\vspace{.5cm}
\end{minipage}
\hfill \vspace{-.5cm}
\begin{minipage}[H]{.49\textwidth}
\begin{algorithm}[H]
  \caption{Gradient Estimation}
  \label{alg:geste}
  \begin{algorithmic}[1]
    \STATE \textbf{Input}: Data Points $\bm{Z} \in \mb{R}^{k\times d}$, Current point $x$

    \STATE $d^*$ = Distance Estimation$(\bm{Z}, x)$
    \STATE $(b, g) = \mte(x,d^*,\bm{Z})$
      \STATE \textbf{Return: }$g$
  \end{algorithmic}
\end{algorithm}
\vspace{.5cm}
\end{minipage}
Therefore, moving a small distance along the vector $v$ should intuitively take us closer to the mean. Given solutions to the polynomial optimization problem \ref{eq:mte}, we may iteratively improve our estimate until no further change is necessary.

\subsection{Algorithm}
\label{sec:alge}
In this section we put the intuition provided previously into practice and propose  a procedure that estimates the mean in the ideal situation where \ref{eq:mte} can be exactly solved (the method is formally described in Algorithm~\ref{alg:meste}):

\begin{enumerate}
\item First, following the median of means framework, the samples $X_i$ are divided into $k$ buckets and the mean of the samples   within each bucket is computed as $Z_i=\frac{k}{n}\sum_{j=(i-1)n/k}^{in/k} X_j$. 
\item Second, the estimate of the mean is iteratively updated using a descent approach, based on the solution of \ref{eq:mte}. As mentioned in Section~\ref{sec:mte}, we need to run \ref{eq:mte} with an appropriate choice of $r$ for the solution $v$ to be correlated with the direction $x-\mu$. In the Distance Estimation step of our algorithm, we estimate a suitable choice of $r$ (see Algorithm~\ref{alg:deste}). This value of $r$ is subsequently used in the Gradient Estimation step, to obtain an appropriate descent direction $g$ (see Algorithm~\ref{alg:geste}). 
\end{enumerate}
From this point on, we refer to the solution of polynomial equations~\ref{eq:mte} as $(b, v) = \mte(x,r,\bm{Z})$.

\subsection{Analysis warm-up} 

In this simplified setting, we provide an analysis of our method and show that it obtains the same guarantees as those presented in Theorem~\ref{thm:sgmest}.  This is formally expressed in the following theorem for Algorithm~\ref{alg:meste} instantiated with Algorithms~\ref{alg:deste} and~\ref{alg:geste}.
\begin{theorem}
\label{thm:sgmeste}
Let $\bm{X} = (X_1, \dots, X_n) \in \mb{R}^{n \times d}$ be $n$ i.i.d.~random vectors with mean $\mu$ and covariance $\Sigma$. Then Algorithm~\ref{alg:meste} instantiated with Algorithms~\ref{alg:deste} and~\ref{alg:geste} and run with inputs $\bm{X}$, target confidence $\delta$, stepsize $\gamma=1/4$ and number of iterations $T = \nite$ returns a vector $\optx$ satisfying:

\begin{equation*}
  \norm{\optx - \mu} \leq \max\lprp{\epsilon, \finrete},
\end{equation*}
with probability at least $1 - \delta$.
\end{theorem}

The main steps involved in the proof are the following:
\begin{enumerate}
  \item \textbf{Distance Estimation:} We show that the Distance Estimation step in Algorithm~\ref{alg:deste} provides an accurate estimate of the distance of the current point from the mean. See Lemma~\ref{lem:deste}.
  \item \textbf{Gradient Estimation:} Next, we show that when $x$ is far away from the mean $\mu$, the vector $g$ obtained by solving ~\ref{eq:mte} in Algorithm~\ref{alg:geste} is well aligned with the vector joining the current point $x$ to the mean $\mu$. See Lemma~\ref{lem:geste}.
  \item \textbf{Gradient Descent:} Combining the previous two steps, we prove that we eventually converge to a good approximation to the mean.
\end{enumerate}
In the proofs of our lemmas relating to the correctness of the Distance Estimation and the Gradient Estimation steps, we make use of the following assumption:
\begin{assumption}
    \label{as:exact}
    For the bucket means, $\bm{Z} = (Z_1, \dots, Z_k)$, we have:
    \begin{equation*}
        \forall v \in \mb{R}^{d}, \norm{v} = 1\;\Rightarrow \abs*{\{i: \inp{Z_i - \mu}{v} \geq \rad\}} \leq 0.05k
    \end{equation*}
\end{assumption}
The assumption is a formalization of the insight of (\cite{lugosi2017sub}), which shows that along all directions, $v$, most of the bucket means are within a small radius of the true mean, $\mu$, with high probability\footnote{This will be made precise in Corollary~\ref{cor:dconce}.}.

First, we prove that the \textbf{Distance Estimation} step defined in Algorithm~\ref{alg:deste} is correct.

\begin{lemma}
  \label{lem:deste}
Under Assumption~\ref{as:exact}, for all $t \in \{0, \dots, T\}$ in the running of Algorihm~\ref{alg:meste}, $\dt{t}$ satisfies:

  \begin{equation*}
    \big\vert \dt{t}-\norm{\xt{t} - \mu}\big\vert  \leq \rade.
  \end{equation*}
\end{lemma}
\begin{proof}
  Let $r^* = \rade$. We first prove the lower bound $\norm{\xt{t} - \mu}-r^*\leq \dt{t}$. We may assume that $\norm{\xt{t} - \mu} > r^*$, as the alternate case is trivially true. For $r = \norm{\xt{t} - \mu} - r^*$, we can simply pick the vector $v = \dd$ where $\dd$ is the unit vector in the direction of $\mu - \xt{t}$. Under Assumption~\ref{as:exact}, we have that for at least $0.95k$ points:
    \begin{equation*}
      \inp{Z_i - \xt{t}}{v} = \inp{Z_i - \mu}{v} + \inp{\mu - \xt{t}}{v} \geq \norm{\xt{t} - \mu} - r^* = r.
    \end{equation*}
  This implies the lower bound holds in the case where $\norm{\xt{t} - \mu} > r$.

For the upper bound $\dt{t}\leq \norm{\xt{t} - \mu}+r^* $, suppose, for the sake of contradiction, there is a value of $r >\norm{\xt{t} - \mu}+r^*$ for which the optimal value of $\mte(\xt{t},r,\bm{Z})$ is greater than $0.9k$. Let $v$ be the solution of $\mte(\xt{t},r,\bm{Z})$. This means that for $0.9k$ of the $Z_i$, we have:

  \begin{equation*}
    \inp{Z_i - \mu}{v} = \inp{Z_i - \xt{t}}{v} + \inp{\xt{t} - \mu}{v} \geq r - \norm{\xt{t} - \mu} > r^*.
  \end{equation*}
  This contradicts Assumption~\ref{as:exact} and proves the upper bound.
\end{proof}
Next, we prove the correctness of the \textbf{Gradient Estimation} step from Algorithm~\ref{alg:geste}.
\begin{lemma}
  \label{lem:geste}
  In the running of Algorithm~\ref{alg:meste}, let us assume $\xt{t}$ satisfies:
  \begin{equation}\label{eq:assummu}
    \norm{\mu - \xt{t}} \geq \radae,
  \end{equation}
and let $\dd$ denote the unit vector in the direction of $\mu - \xt{t}$. Then, under Assumption~\ref{as:exact}, we have that:
  \begin{equation*}
    \inp{\vt{t}}{\dd} \geq \frac{1}{2}.
  \end{equation*}
\end{lemma}
\begin{proof}
  Let $r^* = \rade$.  We have, from the definition of $\dt{t}$, that for $0.9k$ of the $Z_i$,     \mbox{$\inp{Z_i - \xt{t}}{\vt{t}} \geq \dt{t}$}.
  We also have, under Assumption~\ref{as:exact}, that
  $
    \inp{Z_i - \mu}{\vt{t}} \leq r^*
  $ for $0.95k$ of the $Z_i$.
  From the pigeonhole principle, there exists a $Z_j$ which satisfies both those inequalities. Therefore, for that $Z_j$, the lower bound from Lemma~\ref{lem:deste} implies
  \begin{equation*}
    \norm{\mu - \xt{t}} - r^* \leq \dt{t} \leq \inp{Z_j - \xt{t}}{\vt{t}} = \inp{Z_j - \mu}{\vt{t}} + \inp{\mu - \xt{t}}{\vt{t}} \leq r^* + \norm{\mu - \xt{t}}\inp{\dd}{\vt{t}}.
  \end{equation*}
  By rearranging the above inequality and using the assumption on $\norm{\mu - \xt{t}}$ in \eq{assummu},
  we get the required conclusion.
\end{proof}
To control the probability that Assumption~\ref{as:exact} holds, we assume the correctness of the following corollary of Lemma~\ref{lem:mnConc}, formalizing the insight of (\cite{lugosi2017sub}):
\begin{corollary}
\label{cor:dconce}
Let $\bm{Y} = (Y_1, \dots, Y_k) \in \mb{R}^{k \times d}$ be $k$ i.i.d.~random vectors with mean $\mu$ and covariance $\Lambda$. Furthermore, assume $k \geq \nbu$. Then we have for all $v\in \mb{R^d}$ such that $\Vert v\Vert=1$:
\begin{equation*}
 \abs*{\{i: \inp{Y_i - \mu}{v} \geq \rady\}} \leq 0.05k{}
\end{equation*}
{}
with probability at least $1 - \delta$.

\end{corollary}
By instantiating Corollary~\ref{cor:dconce} with the $Y_i = Z_i$, we see that Assumption~\ref{as:exact} holds with high probability. 

Finally, we put the results of Lemma~\ref{lem:deste}, Lemma~\ref{lem:geste} and Corollary~\ref{cor:dconce} together to prove Theorem~\ref{thm:sgmeste}.

\begin{proof}[Proof of Theorem~\ref{thm:sgmeste}]
  Assume first that Assumption~\ref{as:exact} holds. Let $r^* = \radae$. To start with, let us define the set ${\mathcal{G} = \{x: \norm{x - \mu} \leq r^*\}}$. We prove the theorem in two cases:
  \begin{enumerate}
    \item[] \textbf{Case 1: } None of the iterates $x_t$ lie in $\mathcal{G}$. In this case, note that by Lemma~\ref{lem:deste} and the definition of $r^*$, we have:
    \begin{equation}
    \label{eq:dbe}
        \frac{3}{4} \norm{\xt{t} - \mu} \leq \dt{t} \leq \frac{5}{4} \norm{\xt{t} - \mu}.
    \end{equation}
    Moreover, we have by the definition of the update rule of $x_t$ in Algorithm~\ref{alg:meste}:
    \begin{align*}
      \norm{\xt{t + 1} - \mu}^2 &= \norm{\xt{t} - \mu}^2 + \frac{1}{2} \dt{t} \inp{\xt{t} - \mu}{\vt{t}} + \frac{\dt{t}^2}{16} \leq \norm{\xt{t} - \mu}^2 - \frac{\dt{t} \norm{\xt{t} - \mu}}{4} + \frac{\dt{t}^2}{16} \\
      &\leq \norm{\xt{t} - \mu}^2 - \frac{3}{16} \norm{\xt{t} - \mu}^2 + \frac{25}{256} \norm{\xt{t} - \mu}^2 \leq \frac{23}{25} \norm{\xt{t} - \mu}^2,
    \end{align*}
    where we have used Lemma~\ref{lem:geste} for the first inequality and the inequalities in \eq{dbe} for the second inequality. By iteratively applying the above inequality, we get the conclusion of the theorem in this case. 
    \item[] \textbf{Case 2: } At least one of the iterates $\xt{t}$ lies in $\mathcal{G}$. Therefore, we have from Lemma~\ref{lem:deste}:
    \begin{equation*}
      \dt{t} \leq \gpdbe.
    \end{equation*}
    We also have at the completion of the algorithm, from another application of Lemma~\ref{lem:deste}:
    \begin{equation*}
      \norm{\optx - \mu} - \rade \leq \optd \leq \dt{t} \leq  \gpdbe.
    \end{equation*}
      By re-arranging the above inequality, we get the desired result.
  \end{enumerate}
  By Corollary~\ref{cor:dconce}, Assumption~\ref{as:exact} holds with probability at least $1-\delta$ and therefore, the conclusions from Case 1 and Case 2 hold with probability $1-\delta$. 
\end{proof}
Bearing in mind that the polynomial optimization problem \ref{eq:mte} is non-convex, we consider a convex relaxation in the following section.  

\section{Efficient Algorithm for Mean Estimation}
\label{sec:efficient}
In this section, we define a semi-definite programming relaxation of the polynomial optimization problem \ref{eq:mte}. We then design new Distance Estimation and Gradient Estimation algorithms that use the tractable solutions to the relaxation instead of the original polynomial optimization problem. We then use these solutions to update our mean estimate along the same lines as those from Section~\ref{sec:intuition}, albeit with some added technical difficulty. Finally, we provide the analysis of the method and prove Theorem~\ref{thm:sgmest}.

\subsection{The Semi-Definite Relaxation of \ref{eq:mte}}
Here, we propose a semidefinite programming relaxation of \ref{eq:mte}, a variant of the Threshold-SDP from (\cite{hopkins2018sub}). We first define a semidefinite matrix $X \in \mb{R}^{(k + d + 1) \times (k + d + 1)}$ symbolically \newpage
\begin{algorithm}[H]
  \caption{Distance Estimation}
  \label{alg:dest}
  \begin{algorithmic}[1]
    \STATE \textbf{Input}: Data Points $\bm{Z} \in \mb{R}^{k\times d}$, Current point $x$
    \STATE $d^* = \argmax_{r > 0} \mt(x,r,\bm{Z}) \geq 0.9 k$
    \STATE \textbf{Return: } $d^*$
  \end{algorithmic}
\end{algorithm}
\vspace{-.5cm}
\begin{algorithm}[H]
  \caption{Gradient Estimation}
  \label{alg:gest}
  \begin{algorithmic}[1]
    \STATE \textbf{Input}: Data Points $\bm{Z} \in \mb{R}^{k\times d}$, Current point $x$
    \STATE $d^*$ = Distance Estimation$(\bm{Z}, x)$
    \STATE $(X, m) = \mt(x,d^*,\bm{Z})$
    \STATE $X_v = \text{Submatrix of $X$}$ corresponding to the indices $v_i$
    \STATE $g = \text{Top singular vector of $X_{v}$}$
    \STATE $\mathcal{H} = \{i: \inp{Z_i - x}{g} \geq 0\}$
    \IF {$\abs{\mathcal{H}} \geq 0.9k$}
      \STATE \textbf{Return: }$g$
    \ELSE
      \STATE \textbf{Return: }$-g$
    \ENDIF
  \end{algorithmic}
\end{algorithm}
\noindent indexed by $1$, the variables $b_i$ and $v_j$ and denote by the vector $v_{b_i} \coloneqq (X_{b_i, v_1}, \dots, X_{b_i, v_d})$:
\begin{gather*}
  \max \sum_{i = 1}^k X_{1, b_i} \\
  X_{1, b_i} = X_{b_i, b_i} \\
  X_{1,1} = 1 \\
  \sum_{j = 1}^d X_{v_j, v_j} = 1 \\
  \inp{v_{b_i}}{Z_i - x} \geq X_{b_i,b_i}r \ \forall i \in [k]\\
  X \succcurlyeq 0 \tag{\textbf{MT}} \label{eq:mt}
\end{gather*}
Similar to the polynomial optimization~\ref{eq:mte}, this optimization problem is also parameterized by a vector $x \in \mb{R}^d$, $r > 0$ and a matrix $\bm{Z} \in \mb{R}^{k \times d}$. We refer to solutions of this program as $(X, m) = \mt (x,r,\bm{Z})$ with $m$ denoting the optimal value and $X$ denoting the optimal solution.

The main contribution of our paper is in showing that the solutions to the relaxed optimization problem~\ref{eq:mte} can be used to improve the mean estimate similar to those of \ref{eq:mt}.

\subsection{Algorithm}
\label{sec:algorelax}
To efficiently estimate the mean, we instantiate Algorithm~\ref{alg:meste} to use solutions of \ref{eq:mt} instead of \ref{eq:mte}. The new Distance Estimation and Gradient Estimation procedures are stated in Algorithms~\ref{alg:dest} and~\ref{alg:gest}.

As opposed to the polynomial optimization problem, solutions to the relaxation may not necessarily return a single vector $v$ but rather a semidefinite matrix which corresponds to the relaxation of $v$. This matrix may not uniquely determine a direction of improvement. We, therefore, parse the solution to isolate a provably good direction of improvement and use this to iteratively improve our estimate. It is noteworthy that the singular value decomposition does not provide a sign direction. Thankfully the correct orientation is easily ascertained using the data points.

To analyze the runtime of Algorithm~\ref{alg:meste} with Algorithms~\ref{alg:dest} and ~\ref{alg:gest}, we first note that the semidefinite relaxation has $O(k^2 + d^2)$ variables. However, by projecting all the data down to a subspace containing the $k$ bucket means, we may effectively reduce the number of variables to $O(k^2)$ with an $O(k^2d)$ time pre-processing step. Therefore, we are now left with $O(k^2)$ variables. The runtime of interior point methods for solving semidefinite programs with $O(k^2)$ variables and $O(k)$ constraints is $O(k^{3.5})$ (\cite{alizadeh1995interior}). Furthermore, a single call of the Distance Estimation procedure can be efficiently implemented using $\widetilde{O} (1)$ rounds of binary search on the parameter $r$. Therefore, the total cost of a single call to Algorithm~\ref{alg:dest} is $\widetilde{O} (k^{3.5})$. Similarly, the total cost of a call to Algorithm~\ref{alg:gest} is $\widetilde{O} (k^{3.5})$. Since the cost of each iteration is dominated by a single call of Algorithm~\ref{alg:dest} and \ref{alg:gest}, the total cost per iteration is $\widetilde{O} (k^{3.5})$. Since, we only run $\widetilde{O}(1)$ iterations, the total cost of the Algorithm~\ref{alg:meste} instantiated with Algorithms~\ref{alg:dest} and ~\ref{alg:gest} is $\widetilde{O} (k^{3.5} + k^2d)$.

\subsection{Analysis}
We now prove Theorem~\ref{thm:sgmest}. We follow the same lines as the proof of Theorem~\ref{thm:sgmeste}, but  with the added technical difficulties arising from the use of the semi-definite relaxation. 
\begin{enumerate}
  \item \textbf{Distance Estimation:} We show that the Distance Estimation step in Algorithm~\ref{alg:dest} provides an accurate estimate of the distance of the current point from the mean. See Section~\ref{sec:disest}.
  \item \textbf{Gradient Estimation:} Next, we show that when $x$ is far away from the mean $\mu$, the vector $g$ output by Algorithm~\ref{alg:gest} is well aligned with the vector joining the current point $x$ to the mean $\mu$. See Section~\ref{sec:gradest}.
  \item \textbf{Gradient Descent:} Combining the previous two steps, we prove that we eventually converge to a good approximation to the mean. See Section~\ref{sec:smTh}.
\end{enumerate}
The following assumption is required to prove the correctness of the Distance Estimation and Gradient Estimation steps:
\begin{assumption}
    \label{as:relax}
    For the bucket means, $\bm{Z} = (Z_1, \dots, Z_k)$, let $\mathcal{S}_r$ denote the set of feasible solutions for $\mt (\mu, r, \bm{Z})$. Then, we have for all $r \geq \rad$,
    \begin{equation*}
        \max_{X \in \mathcal{S}_r} \sum_{i = 1}^k X_{b_i,b_i} \leq \frac{k}{20}.
    \end{equation*}
\end{assumption}
The above assumption is a strengthening of Assumption~\ref{as:exact} for the case where we use \ref{eq:mt} instead of \ref{eq:mte}. We use the following fact at several points in the subsequent analysis:
\begin{remark}
  \label{rem:a21}
  Note that Assumption~\ref{as:relax} implies Assumption~\ref{as:exact}.
\end{remark}

\subsubsection{Distance Estimation Step}
\label{sec:disest}
In this subsection, we analyze the Distance Estimation step from Algorithm~\ref{alg:dest}. We  show that an accurate estimate of the distance of the current point from the mean can be found. We begin by stating a lemma that shows that a feasible solution for $\mt(x,r,\bm{Z})$ can be converted to a feasible solution for $\mt(\mu,\rad,\bm{Z})$ with a reduction in optimal value. 
\begin{lemma}
  \label{lem:conv}
 Let us assume Assumption~\ref{as:relax}.  Let $X \in \mb{R}^{(k + d + 1) \times (k + d + 1)}$ be a positive semi-definite matrix, symbolically indexed by $1$ and the variables $b_i$ and $v_j$. Moreover, suppose that $X$ satisfies:
  \begin{equation*}
    X_{1,1} = 1, \quad X_{b_i, b_i} = X_{1, b_i}, \quad \sum_{j = 1}^d X_{v_j, v_j} = 1, \quad \sum_{i = 1}^k X_{b_i, b_i} \geq 0.9k.
  \end{equation*}
  Then, there is a set of at least $0.85k$ indices $\mathcal{T}$ such that for all $i \in \mathcal{T}$:
  \begin{equation*}
    \inp{Z_i - \mu}{v_{b_i}} < X_{b_i, b_i} \rad,
  \end{equation*}
  and a set of at least $k / 3$ indices $\mathcal{R}$ such that for all $j \in \mathcal{R}$, we have $X_{b_j, b_j} \geq 0.85$.
\end{lemma}
\begin{proof}
  Let $r = \rad$. We prove the lemma by contradition. Firstly, note that $X$ is infeasible for $\mt(\mu, r, \bm{Z})$ as the optimal value for $\mt(\mu,r,\bm{Z})$ is less than $k / 20$ (Assumption~\ref{as:relax}). Note that the only constraints of $\mt(\mu,r,\bm{Z})$ that are violated by $X$ are constraints of the form:
  \begin{equation*}
    \inp{Z_i - \mu}{v_{b_i}} < X_{b_i, b_i} r.
  \end{equation*}
Now, let $\mathcal{T}$ denote the set of indices for which the above inequality is violated. We can convert $X$ to a feasible solution for $\mt(\mu, r, \bm{Z})$ by setting to $0$ the rows and columns corresponding to the indices in $\mathcal{T}$. Let $X^\prime$ be the matrix obtained by the above operation. We have from Assumption~\ref{as:relax}:
  \begin{equation*}
    0.05k \geq \sum_{i = 1}^k X^\prime_{b_i, b_i} = \sum_{i = 1}^k X_{b_i, b_i} - \sum_{i \in \mathcal{T}} X_{b_i, b_i} \geq 0.9k - \abs{\mathcal{T}},
  \end{equation*}
where the last inequality follows from the fact that $X_{b_i, b_i} \leq 1$. By rearranging the above inequality, we get the first claim of the lemma.

For the second claim, let $\mathcal{R}$ denote the set of indices $j$ satisfying $X_{b_j, b_j} \geq 0.85$. We have:
\begin{equation*}
    0.9k \leq \sum_{j = 1}^k X_{b_j, b_j} = \sum_{j \in \mathcal{R}} X_{b_j, b_j} + \sum_{j \notin \mathcal{R}} X_{b_j, b_j} \leq \abs{\mathcal{R}} + 0.85k - 0.85\abs{\mathcal{R}} \implies \frac{k}{3} \leq \abs{\mathcal{R}}.
  \end{equation*}
  This establishes the second claim of the lemma.
\end{proof}
The following  lemma shows that if the distance between the mean $\mu$ and a point $x$ is small then the estimate returned by Algorithm~\ref{alg:dest} is also small.
\begin{lemma}
\label{lem:destg}
  Suppose a point $x\in\mb{R}^d$ satisfies $\norm{x - \mu} \leq \radAlg$. Then, under Assumption~\ref{as:relax}, Algorithm~\ref{alg:dest} returns a value $d^\prime$ satisfying
  \begin{equation*}
    d^\prime \leq \radAlgG.
  \end{equation*}
\end{lemma}
\begin{proof}
Let $r^\prime = \radAlgG$ and $r = \rad$. Suppose that the optimal value of $\mt(x,r^\prime,\bm{Z})$ is greater than $0.9k$ and let its optimal solution be $X$. Let $\mathcal{R}$ and $\mathcal{T}$ denote the two sets whose existence is guaranteed by Lemma~\ref{lem:conv}. From, the cardinalities of $\mathcal{R}$ and $\mathcal{T}$, we see that their intersection is not empty. For $j \in \mathcal{R} \cap \mathcal{T}$, we have:
  \begin{equation*}
    0.85r^\prime \leq \inp{Z_j - x}{v_{b_j}} = \inp{Z_j - \mu}{v_{b_j}} + \inp{\mu - x}{v_{b_j}} < r + \norm{x - \mu},
  \end{equation*}
where the first inequality follows from the fact that $j \in \mathcal{R}$ and the fact that $X$ is feasible for $\mt(x, r^\prime, \bm{Z})$ and the last inequality follows from the inclusion of $j$ in $\mathcal{T}$ and Cauchy-Schwarz.

By plugging in the bounds on $r^\prime$ and $r$, we get:
  \begin{equation*}
    \norm{x - \mu} > 6075 \lprp{\sqrt{\Tr \Sigma/n} + \sqrt{k\norm{\Sigma} / n}}.
  \end{equation*}
  This contradicts the assumption on $\norm{x - \mu}$ and concludes the proof of the lemma.
\end{proof}
The next  lemma shows that the distance between the mean $\mu$ and a point $x$ can be accurately estimated as long as $x$ is sufficiently far from $\mu$.
\begin{lemma}
\label{lem:dest}
  Suppose a point $x$ satisfies $\tilde{d} = \norm{x - \mu} \geq \radAlg$. Then, under Assumption~\ref{as:relax}, Algorithm~\ref{alg:dest} returns a value $d^\prime$ satisfying:
  \begin{equation*}
    0.95\tilde{d} \leq d^\prime \leq 1.25\tilde{d}.
  \end{equation*}
\end{lemma}
\begin{proof}
  Let us define the direction $\dd$ to be the unit vector in the direction of $x - \mu$. From Assumption~\ref{as:exact} (which is implied by Assumption~\ref{as:relax}), the number of $Z_i$ satisfying $\inp{Z_i - \mu}{\dd} \geq \rad$ is less than $k / 20$. Therefore, we have that for at least $0.95k$ points:
  \begin{equation*}
    \inp{Z_i - x}{-\dd} = \inp{x - \mu + \mu - Z_i}{\dd} = \norm{x - \mu} - \rad \geq 0.95\tilde{d}.
  \end{equation*}
Along with the monotonicity\footnote{See Lemma~\ref{lem:mtmon} in Appendix~\ref{sec:aux}.} of $\mt(x,r,\bm{Z})$ in $r$, this implies the lower bound.

For the upper bound, we show that the optimal value of $\mt(x,1.25\tilde{d},\bm{Z})$ is less than $0.9k$. For the sake of contradiction, suppose that this optimal value is greater than $0.9k$. Let $X$ be a feasible solution of $\mt(x, 1.25\tilde{d}, \bm{Z})$ that achieves $0.9k$. Let $\mathcal{R}$ and $\mathcal{T}$ be the two sets whose existence is guaranteed by Lemma~\ref{lem:conv} and $j$ be an element in their intersection. We have for $j$:
  \begin{align*}
    0.85 (1.25\tilde{d}) &\leq X_{b_j,b_j} 1.25\tilde{d}\leq \inp{Z_j - x}{v_{b_j}} = \inp{Z_j - \mu}{v_{b_j}} \!+\! \inp{\mu - x}{v_{b_j}}  \\
    &< X_{b_j, b_j} \rad \!+\! \norm{\mu \!-\! x} 
    \!=\! X_{b_j, b_j}\rad \!+\! \tilde{d},
  \end{align*}
  where the first inequality follows from the inclusion of $j$ in $\mathcal{R}$ and the last inequality follows from the inclusion of $j$ in $\mathcal{T}$ and Cauchy-Schwarz.
By re-arranging the above inequality, we get:
  \begin{equation*}
    X_{b_j, b_j} > (1.0625\tilde{d} - \tilde{d})\Big(\rad\Big)^{-1} > 1,
  \end{equation*}
which is a contradiction. Therefore, we get from the monotonicity of $\mt(x,r,\bm{Z})$ (see Lemma~\ref{lem:mtmon}), that $d^\prime \leq 1.25\tilde{d}$ and this concludes the proof of the lemma.
\end{proof}

\subsubsection{Gradient Estimation Step}\label{sec:gradest}
In this section, we analyze the Gradient Estimation step of the algorithm. We show that an approximate gradient can be found as long as the current point $x$ is not too close to the mean $\mu$. 
The following lemma shows that we obtain a non-trivial estimate of the gradient in Algorithm~\ref{alg:gest}.
\begin{lemma}
  \label{lem:gest}
  Suppose a point $x$ satisfies $\norm{x - \mu} \geq \radAlg$ and let $\dd$ be the unit vector along $\mu - x$. Then under Assumption~\ref{as:relax}, Algorithm~\ref{alg:gest} returns a vector $g$ satisfying:
  \begin{equation*}
    \inp{g}{\dd} \geq \frac{1}{15}.
  \end{equation*}
\end{lemma}
\begin{proof}
In the running of Algorithm~\ref{alg:gest}, let $X$ denote the solution of $\mt(x,d^*,\bm{Z})$. We begin by factorizing the solution $X$ into $UU^\top$ with the rows of $U$ denoted by $u_1$, $u_{b_1}, \dots, u_{b_k}$ and $u_{v_1}, \dots, u_{v_d}$. We also define the matrix $U_v \!=\! (u_{v_1}, \dots, u_{v_d})$ in $\mb{R}^{(k + d + 1) \times d}$. From the constraints in \ref{eq:mt}, we~have:
  \begin{equation*}
    X_{b_i, b_i} = \norm{u_{b_i}}^2 \leq 1 \implies \norm{u_{b_i}} \leq 1,\quad \sum_{j = 1}^d X_{v_j, v_j} = \sum_{j = 1}^d \norm{u_{v_j}}^2 = \norm{U_v}_F^2 = 1 \implies \norm{U_v}_F = 1.
  \end{equation*}
Let $\mathcal{R}$ and $\mathcal{T}$ denote the sets defined in Lemma~\ref{lem:conv}. Let $j \in \mathcal{T} \cap \mathcal{R}$. By noting that $v_{b_j} = u_{b_j}^\top U_v$, we have for $j$:
  \begin{equation*}
     0.85d^* \leq \inp{Z_j - \mu}{v_{b_j}} + \inp{\mu - x}{v_{b_j}} \leq X_{b_j, b_j} \rad + u_{b_j}^\top U_v (\mu - x),
  \end{equation*}
where the first inequality follows from the inclusion of $j$ in $\mathcal{R}$ and the second from its inclusion in $\mathcal{T}$. We get by rearranging the above equation and using our bound on $d^*$ from Lemma~\ref{lem:dest}:
  \begin{equation}\label{eq:aboveeq}
     0.80 \norm{\mu - x} \leq 0.85d^* \leq X_{b_j, b_j} \rad + u_{b_j}^\top U_v (\mu - x).
  \end{equation}
By rearranging \eq{aboveeq}, using Cauchy-Schwarz, $\norm{u_{b_i}} \leq 1$ and the assumption on $\norm{x - \mu}$:
  \begin{equation*}
    \norm{U_v (\mu - x)} \geq u_{b_j}^\top U_v (\mu - x) \geq 0.75 \norm{\mu - x}.
  \end{equation*}
We finally get that:
  \begin{equation*}
    \norm{U_v \dd} \geq 0.75.
  \end{equation*}
  Now, we have:
  \begin{equation*}
    1 = \norm{U_v}_F^2 = \norm{U_v \proj_\dd}_F^2 + \norm{U_v \projp_\dd}_F^2 \geq \norm{U_v \projp_\dd}_F^2 + (0.75)^2 \implies \norm{U_v \projp_\dd}_F \leq 0.67.
  \end{equation*}
 Let $y$ be the top singular vector of $X_v$. Note that $X_v = U_v^\top U_v$ and $y$ is also the top right singular vector of $U_v$. We have that:
  \begin{equation*}
    0.75 \leq \norm{U_v y} \leq \norm{U_v\proj_\dd y} + \norm{U_v\projp_\dd y} \leq \norm{\proj_\dd y} + \norm{U_v\projp_\dd}_F \leq \norm{\proj_\dd y} + 0.67.
  \end{equation*}
 This means that we have:
  \begin{equation*}
    \abs{\inp{y}{\dd}} \geq \frac{1}{15}.
  \end{equation*}
Note that the algorithm returns either $y$ or $-y$. Firstly, consider the case where $\inp{y}{\dd} > 0$. From Assumption~\ref{as:exact} (implied by Assumption~\ref{as:relax}), we have for at least $0.95k$ points:
  \begin{equation*}
    \inp{Z_i - \mu}{y} \leq \rad.
  \end{equation*}
  Therefore, we have for $0.95k$ points:
  \begin{align*}
    \!\!\inp{Z_i - x}{y}\! &=\! \inp{Z_i - \mu}{y} + \inp{\mu - x}{y} \\
    &\geq\! - \rad +\frac{\radAlg}{15} > 0.
  \end{align*}
This means that in the case where $\inp{y}{\dd} > 0$, we return $y$ which satisfies $\inp{\mu - x}{y} > 0$. This implies the lemma in this case. The case where $\inp{y}{\dd} < 0$ is similar with $-y$ used instead of $y$. This concludes the proof of the lemma.
\end{proof}

\subsubsection{Gradient Descent Step}
\label{sec:smTh}

The following lemma guarantees that Assumption~\ref{as:relax} holds with high probability and is used analogously to Corollary~\ref{cor:dconce} in the proof of Theorem~\ref{thm:sgmeste}:
\begin{lemma}
  \label{lem:mnConc}
  Let $\bm{Y} = (Y_1, \dots, Y_k) \in \mb{R}^{k \times d}$ be $k$ i.i.d.~random vectors with mean $\mu$ and covariance $\Lambda$ and let $\mathcal{S}$ denote the set of feasible solutions of $\mt(\mu, r, \bm{Y})$. Then, we have for $r \geq \rady$ and $k \geq \nbu$:
  \begin{equation*}
    \max_{X \in \mathcal{S}} \sum_{i = 1}^k X_{b_i, b_i} \leq \frac{k}{20},
  \end{equation*}
  with probability at least $1 - \delta$.
\end{lemma}
The proof of the lemma is an application of standard empirical process theory and concentration inequalities (\cite{lugosi2017sub,hopkins2018sub}) and is proven in Appendix~\ref{sec:proofcon}.

The rest of the proof of Theorem~\ref{thm:sgmest} follows the same lines as that of Theorem~\ref{thm:sgmeste} and is postponed to Appendix~\ref{sec:proofsmth}. 
\section{Conclusion}
In this paper, we proposed a computationally efficient estimator for the mean of a random vector which obtains the statistically optimal performance. This estimator has a significantly faster runtime together with a simpler analysis than previous works. Our algorithm is based on a descent method, where a current estimate of the mean is iteratively improved. 

Considering the extension to M-estimation procedures~(\cite{brownlees2015, HsuSab16,lugosi2018risk}) is a promising direction for further research, with as first step, the particular example of linear regression with heavy tailed noise and covariates~(\cite{audibert2011}).

\bibliography{refs}
\bibliographystyle{alpha}
\newpage
\appendix
\section{Auxiliary lemma}
\label{sec:aux}
\begin{lemma}
  \label{lem:mtmon}
  For any $\bm{Z} \in \mb{R}^{k \times d}$ and $x \in \mb{R}^d$, the optimal value of $\mt(x,r,\bm{Z})$ is monotonically non-increasing in $r$.
\end{lemma}
\begin{proof}
  The lemma follows trivially from the fact that a feasible solution $X$ of $\mt(x,r,\bm{Z})$ is also a feasible solution for $\mt(x,r^\prime,\bm{Z})$ for $r^\prime \leq r$.
\end{proof}

\section{Proof of Lemma~\ref{lem:mnConc}}
\label{sec:proofcon}
We first show that the optimal value of the semi-definite program \ref{eq:mt} satisfies a bounded-difference condition with respect to the $Z_i$'s.

\begin{lemma}
  \label{lem:conc}
  Let $\bm{Y} = (Y_1, \dots, Y_k)$ be any set of $k$ vectors in $\mb{R}^d$. Now, let $\bm{Y}^\prime = (Y_1, \dots, Y_i^\prime, \dots, Y_k)$ be the same set of $k$ vectors with the $i^{th}$ vector replaced by $Y_i^\prime\in \mb{R}^d$. If $m$ and $m^\prime$ are the optimal values of $\mt(x,r,\bm{Y})$ and $\mt(x,r,\bm{Y}^\prime)$, we have:

  \begin{equation*}
    \abs{m - m^\prime} \leq 1
  \end{equation*}
\end{lemma}

\begin{proof}
  Firstly, assume that $X$ is a feasible solution to $\mt(x,r,\bm{Y})$. Now, let us define $X^\prime$ as:

  \begin{equation*}
    X^\prime_{i,j} = \begin{cases}
                      X_{i,j} & \text{ if $i,j \neq b_i$} \\
                      0 &\text{otherwise}
                     \end{cases}
  \end{equation*}
  That is $X^\prime$ is equal to $X$ except with the row and column corresponding to $b_i$ being set to $0$. We see that $X^\prime$ forms a feasible solution to $\mt(x,r,\bm{Y}^\prime)$. Therefore, we have that:

  \begin{equation*}
    \sum_{j = 1}^k X_{b_j, b_j} = \sum_{j = 1, j\neq i}^k X^\prime_{b_j, b_j} + X_{b_i, b_i} \leq \sum_{j = 1, j\neq i}^k X^\prime_{b_j, b_j} + 1 \leq m^\prime + 1
  \end{equation*}
where the bound $X_{b_i,b_i} \leq 1$ follows from the fact that the $2\times 2$ sub-matrix of $X$ formed by the rows and columns indexed by $1$ and $b_i$ is positive semidefinite and the constraint that $X_{b_i, b_i} = X_{1, b_i}$. Since the above series of equalities holds for all feasible solutions $X$ of $\mt(x,r,\bm{Y})$, we get:

  \begin{equation*}
    m \leq m^\prime + 1.
  \end{equation*}
  Through a similar argument, we also conclude that $m^\prime \leq m + 1$. Putting the above two inequalities together, we get the required conclusion.
\end{proof}
For the next few lemmas, we are concerned with the case where $x = \mu$. Since we already know that the optimal SDP value satisfies the bounded differences condition, we need to verify that the expectation is small. As a first step towards this, we define the 2-to-1 norm of a matrix $M$.

\begin{definition}
  The 2-to-1 norm of $M \in \mb{R}^{n \times d}$ is defined as

  \begin{equation*}
    \normtto{M} = \max_{\substack{\norm{v} = 1 \\ \sigma_i \in \{\pm 1\}}} \sigma^\top M v = \max_{\norm{v} = 1} \norm{Mv}_1
  \end{equation*}
\end{definition}
We consider the classical semidefinite programming relaxation of the 2-to-1 norm. To start with, we will define a matrix $X \in \mb{R}^{(n + d + 1) \times (n + d + 1)}$ with the rows and columns indexed by $1$ and the elements $\sigma_i$ and $v_j$. The semidefinite programming relaxation is defined as follows:
\begin{gather*}
  \max \sum_{i, j} M_{i,j} X_{\sigma_i, v_j} \\
  X_{1,1} = 1\\
  \sum_{j = 1}^d X_{v_j,v_j} = 1 \\
  X_{\sigma_i,\sigma_i} = 1 \\
  X \succcurlyeq 0 \tag{TOR} \label{eq:tor}
\end{gather*}
We now state a theorem of Nesterov as stated in (\cite{hopkins2018sub}):
\begin{theorem}{(\cite{nesterov1998semidefinite})}
  \label{thm:contr}
  There is a constant $K_{2\rightarrow 1} = \sqrt{\pi/2} \leq 2$ such that the optimal value, $m$, of the semidefinite programming relaxation~\ref{eq:tor} satisfies:
  \begin{equation*}
    m \leq K_{2\rightarrow 1} \normtto{M}.
  \end{equation*}
\end{theorem}
In the next step, we will bound the expected 2-to-1 norm of the random matrix $Z$. To do this, we begin by stating the famous Ledoux-Talagrand Contraction Theorem (\cite{ledoux1991probability}).
\begin{theorem}
  \label{thm:ledtal}
  Let $X_1, \dots, X_n \in \mb{R}^d$ be i.i.d.~random vectors, $\mathcal{F}$ be a class of real-valued functions on $\mb{R}^d$ and $\sigma_i, \dots, \sigma_n$ be independent Rademacher random variables. If $\phi: \mb{R} \rightarrow \mb{R}$ is an $L$-Lipschitz function with $\phi(0) = 0$, then:

  \begin{equation*}
    \mathbb{E} \sup_{f \in \mathcal{F}} \sum_{i = 1}^n \sigma_i \phi(f(X_i)) \leq L\cdot \mb{E} \sup_{f \in \mathcal{F}} \sum_{i = 1}^{n} \sigma_i f(X_i).
  \end{equation*}
\end{theorem}
We are now ready to bound the expected 2-to-1 norm of the random matrix $Z$.
\begin{lemma}
  \label{lem:tto}
  Let $\bm{Y} = (Y_1, \dots, Y_n) \in \mb{R}^{n\times d}$ be a set of $n$ i.i.d.~random vectors such that $\mb{E} [Y_i] = 0$ and $\mb{E} [Y_i Y_i^\top] = \Lambda$. Then, we have:
  \begin{equation*}
    \mb{E} \normtto{\bm{Y}} \leq 2 \sqrt{n\Tr \Lambda} + n \norm{\Lambda}^{1/2}.
  \end{equation*}
\end{lemma}
\begin{proof}
  Denoting by $Y$ and $Y_i^\prime$ random vectors that are independently and identically distributed as $Y_i$ and by $\sigma_i$ independent Rademacher random variables, we have: 
  \begin{align*}
    \mb{E} [\normtto{\bm{Y}}] &= \mb{E} \lsrs{\max_{\norm{v} = 1}\sum_{i = 1}^n \abs{\inp{Y_i}{v}}} = \mb{E} \lsrs{\max_{\norm{v} = 1}\sum_{i = 1}^n \abs{\inp{Y_i}{v}} + \mb{E} \abs{\inp{v}{Y_i}} - \mb{E} \abs{\inp{v}{Y_i}}}\\
    &\leq \mb{E} \lsrs{\max_{\norm{v} = 1} \sum_{i = 1}^n \abs{\inp{Y_i}{v}} - \mb{E} \abs{\inp{Y_i^\prime}{v}}} + n \max_{\norm{v} = 1}\mb{E} [\abs{\inp{v}{Y}}] \\
    &\leq \mb{E} \lsrs{\max_{\norm{v} = 1} \sum_{i = 1}^n \sigma_i (\abs{\inp{Y_i}{v}} - \abs{\inp{Y_i^\prime}{v}})} + n \max_{\norm{v} = 1} \mb{E} \lsrs{\abs{\inp{v}{Y}}}.
  \end{align*}
  Now, we have for the second term:
  \begin{equation*}
    \max_{\norm{v} = 1}\mb{E} [\abs{\inp{v}{Y}}] \leq \max_{\norm{v} = 1} \sqrt{\mb{E} \inp{v}{Y}^2} \leq \norm{\Lambda}^{1/2}.
  \end{equation*}
  For the first term, we get via a standard symmetrization argument:
  \begin{align*}
    \mb{E} \lsrs{\max_{\norm{v} = 1} \sum_{i = 1}^n \sigma_i (\abs{\inp{Y_i}{v}} - \abs{\inp{Y_i^\prime}{v}})} & \leq \mb{E} \lsrs{\max_{\norm{v} = 1} \sum_{i = 1}^n \sigma_i \abs{\inp{Y_i}{v}}} + \mb{E}\lsrs{\max_{\norm{v} = 1}\sum_{i = 1}^n-\sigma_i\abs{\inp{Y_i^\prime}{v}}} \\
    &= 2 \mb{E} \lsrs{\max_{\norm{v} = 1} \sum_{i = 1}^n \sigma_i\abs{\inp{v}{Y_i}}} \leq 2 \mb{E} \lsrs{\max_{\norm{v} = 1} \sum_{i = 1}^n \sigma_i\inp{v}{Y_i}} \\
    &= 2 \mb{E} \lsrs{\norm*{\sum_{i = 1}^n \sigma_i Y_i}} \leq 2 \lprp{\mb{E} \lsrs{\norm*{\sum_{i = 1}^n \sigma_i Y_i}^2}}^{1/2} \\
    &= 2 \lprp{\mb{E} \sum_{1 \leq i,j \leq n} \sigma_i \sigma_j \inp{Y_i}{Y_j}}^{1/2} = 2\sqrt{n\Tr \Lambda},
  \end{align*}
  where the second inequality follows from the Ledoux-Talagrand Contraction Principle (Theorem~\ref{thm:ledtal})
  By putting the above two bounds together, we get the lemma.
\end{proof}
We now bound the expected value of $\mt(\mu,r,\bm{Y})$ by relating it to $\normtto{\bm{Y}}$.
\begin{lemma}
  \label{lem:exp}
  Let $\bm{Y} = (Y_1, \dots, Y_k) \in \mb{R}^{k\times d}$ be a collection of $k$ i.i.d.~random vectors with mean $\mu$ and covariance $\Lambda$. Now, denoting by $\mathcal{S}$ the set of feasible solutions for $\mt(\mu,r,\bm{Y})$, we have:
  \begin{equation*}
    \mb{E} \max_{x \in \mathcal{S}} \sum_{i = 1}^k X_{1, b_i} \leq \frac{1}{2r} \lprp{5 \sqrt{k \Tr \Lambda} + 2k \norm{\Lambda}^{1/2}}.
  \end{equation*}
\end{lemma}

\begin{proof}
  Firstly, let $X$ be a feasible solution for $\mt(\mu, r, \bm{Y})$. We construct a new matrix $W$ which is indexed by $\sigma_i$ and $v_j$ as opposed to $b_i$ and $v_j$ for $X$:
  \begin{gather*}
    W_{\sigma_i, \sigma_j} = 4X_{b_i, b_j} - 2 X_{1, b_i} - 2 X_{1, b_j} + 1,\quad W_{v_i, v_j} = X_{v_i, v_j},\quad W_{1,1} = 1, \\
    W_{1, v_i} = X_{1, v_i}, \quad W_{1, b_i} = 2X_{1, b_i} - 1, \quad W_{v_i, b_j} = 2X_{v_i, b_j} - X_{1, v_i}.
  \end{gather*}
  We prove that $Y$ is a feasible solution to the SDP relaxation~\ref{eq:tor} of $\bm{Y} - \mu$. We see that:
  \begin{equation*}
    W_{\sigma_i, \sigma_i} = 1 \text{ and } \sum_{i = 1}^{d} W_{v_i, v_i} = 1.
  \end{equation*}
  Then, we simply need to verify that $Y$ is PSD. Let $w \in \mb{R}^{k + d + 1}$ indexed by $1$, $\sigma_i$ and $v_j$. We construct from $w$ a new vector $w^\prime$, indexed by $1$, $b_i$ and $v_j$ and defined as follows:
  \begin{equation*}
    w^\prime_{1} = w_{1} - \sum_{i = 1}^k w_{\sigma_i}, \quad w^\prime_{b_i} = 2 w_{\sigma_i}, \quad w^\prime_{v_j} = w_{v_j}.
  \end{equation*}
  With $w^\prime$ defined as above, we have the following equality:
  \begin{equation*}
    w^\top W w = (w^\prime)^\top X w^\prime \geq 0.
  \end{equation*}
  Since the above condition holds for all $w \in \mb{R}^{k + d + 1}$, we get that $Y \succcurlyeq 0$. Therefore, we conclude that $Y$ is a feasible solution to the SDP relaxation~\ref{eq:tor} of $\bm{Y} - \mu$.

  We bound the expected value of $\mt(\mu,r,\bm{Y})$ as follows, denoting by $v_{b_i}$ the vector $(X_{b_i, v_1}, \dots, X_{b_i, v_d})$ and by $v$ the vector $(X_{1, v_1}, \dots, X_{1, v_d})$:
  \begin{align*}
    \mb{E} \max_{X \in \mathcal{S}} \sum_{i = 1}^k X_{1, b_i} &= \mb{E} \max_{X \in \mathcal{S}} \sum_{i = 1}^k X_{b_i, b_i} \leq \frac{1}{r}\mb{E} \max_{X \in \mathcal{S}} \sum_{i = 1}^k \inp{v_{b_i}}{Y_i - \mu} \\
    &= \frac{1}{2r} \mb{E} \max_{X \in S}\Big[ \sum_{i = 1}^k \inp{2v_{b_i} - v}{Y_i - \mu} + \sum_{i = 1}^k \inp{v}{Y_i - \mu}\Big] \\
    &\leq \frac{1}{2r} \lprp{\mb{E} \max_{X \in S} \sum_{i = 1}^k \inp{2v_{b_i} - v}{Y_i - \mu} + \mb{E} \max_{X \in \mathcal{S}} \sum_{i = 1}^k \inp{v}{Y_i - \mu}}.
  \end{align*}
  We note that from the fact that $X$ is PSD, we have that (from the fact that the $2\times 2$ submatrix indexed by $v_i$ and $b_j$ is PSD):
  \begin{equation*}
    X^2_{v_i,b_j} \leq X_{v_i, v_i} X_{b_j, b_j} \leq X_{v_i, v_i} \implies \norm{v_{b_j}}^2= \sum_{i = 1}^d X^2_{v_i, b_j}  \leq \sum_{i = 1}^d X_{v_i, v_i} = 1.
  \end{equation*}
  Therefore, we get for the second term in the above equation:
  \begin{equation*}
    \mb{E} \max_{X \in \mathcal{S}} \sum_{i = 1}^{k} \inp{v}{Y_i - \mu} \leq \mb{E} \norm*{\sum_{i = 1}^{k} Y_i - \mu} \leq \lprp{\mb{E} \norm*{\sum_{i = 1}^{k} Y_i - \mu}^2}^{1/2} = (k \Tr \Lambda)^{1/2}.
  \end{equation*}
  We bound the first term using the following series of inequalities where $Y$ is constructed from $X$ as described above:
  \begin{align*}
    \mb{E} \max_{x \in \mathcal{S}} \sum_{i = 1}^k \inp{2v_{b_i} - v}{Y_i - \mu} &= \mb{E} \max_{x \in \mathcal{S}} \sum_{i = 1}^k \sum_{j = 1}^d (Y_i - \mu)_j W_{\sigma_i, v_j} = \mb{E} \max_{x \in \mathcal{S}} \sum_{i = 1}^k \sum_{j = 1}^d (\bm{Y}_{i, j} - \mu_j) W_{\sigma_i, v_j} \\
    &\leq 2 \mb{E} \normtto{\bm{Y} - \bm{1}\mu^\top} \leq 4 \sqrt{k \Tr \Lambda} + 2 k \norm{\Lambda}^{1/2},
  \end{align*}
  where the first inequality follows from Theorem~\ref{thm:contr} and the second inequality follows from Lemma~\ref{lem:tto}. By combining the above three inequalities, we finally get:
  \begin{equation*}
    \mb{E} \max_{x \in \mathcal{S}} \sum_{i = 1}^k X_{1, b_i} \leq \frac{1}{2r} \lprp{5 \sqrt{k \Tr \Lambda} + 2 k \norm{\Lambda}^{1/2}}.
  \end{equation*}
\end{proof}  
We are now able to prove Lemma~\ref{lem:mnConc}.

\begin{proof}[Lemma~\ref{lem:mnConc}]
  From Lemma~\ref{lem:exp}, we see that:
  \begin{equation*}
    \mb{E} \max_{X \in \mathcal{S}} \sum_{i = 1}^k X_{b_i, b_i} \leq \frac{k}{40}.
  \end{equation*}
  Now from Lemma~\ref{lem:conc} and an application of the bounded difference inequality (see, for example, Theorem~6.2 in \cite{boucheron2013concentration}), with probability at least $1 - \delta$:
  \begin{equation*}
    \max_{X \in \mathcal{S}} \sum_{i = 1}^k X_{b_i, b_i} \leq \frac{k}{20}.
  \end{equation*}
\end{proof}
\section{Proof of Theorem~\ref{thm:sgmest}}
\label{sec:proofsmth}
Let $\mathcal{G} = \{x: \norm{x - \mu} \leq \radAlg \}$. Also, we assume that Assumption~\ref{as:relax} holds. We prove the theorem differentiating between two cases:
  \begin{enumerate}
    \item[] \textbf{Case 1: } None of the iterates $\xt{t}$ fall into the set $\mathcal{G}$. In this case, we have from Lemma~\ref{lem:dest} that:
    \begin{equation}
    \label{eqn:dbo}
        0.95 \norm{\xt{t} - \mu} \leq \dt{t} \leq 1.25 \norm{\xt{t} - \mu}
    \end{equation}
    Now, we get:
    \begin{align*}
      \norm{\xt{t + 1} - \mu}^2 &= \norm{\xt{t} - \mu}^2 - 2\frac{\dt{t}}{20}\inp{\gt{t}}{\mu - \xt{t}} + \frac{\dt{t}^2}{400} \leq \norm{\xt{t} - \mu}^2 - \frac{\dt{t}\norm{\mu - \xt{t}}}{150} + \frac{\dt{t}^2}{400} \\
      &\leq \norm{\xt{t} - \mu}^2 - \dt{t} \lprp{\frac{\norm{\mu - \xt{t}}}{150} - \frac{\dt{t}}{400}} \leq \lprp{1 - \frac{1}{500}} \norm{\xt{t} - \mu}^2.
    \end{align*}
    where the first inequality follows from Lemma~\ref{lem:gest} and the last inequality follows by substituting the lower bound on $\dt{t}$ in the first term and the upper bound on $\dt{t}$ in the second term (Equation \eqref{eqn:dbo}).
    By an iterated application of the above inequality, we get the required result.
    \item[] \textbf{Case 2: } One of the iterates $\xt{t}$ falls into the set $\mathcal{G}$. If the algorithm returns an element from $\mathcal{G}$, the theorem is trivially true. From Lemma~\ref{lem:destg},  we have for this iterate $\xt{t}\in \mathcal{G}$ that:
    \begin{equation*}
      \dt{t} \leq \radAlgG.
    \end{equation*}
    Therefore, we have at the completion of the algorithm a value $\optd \leq \radAlgG$ together with $x^*$ lying outside $\mathcal{G}$. Thus, we finally have from Lemma~\ref{lem:dest}:
    \begin{equation*}
      0.95\norm{x^* - \mu} \leq \radAlgG \implies \norm{x^* - \mu} \leq \radRet.
    \end{equation*}
  \end{enumerate}
   By Lemma~\ref{lem:mnConc}, Assumption~\ref{as:relax} holds with probability at least $1-\delta$ and therefore, the conclusions from Case 1 and Case 2 hold with probability $1-\delta$.
   
   Substituting the value of $k$, we obtain
     \begin{align*}
    \norm{\optx - \mu} &\leq \max\lprp{\epsilon, \radRet} \\
    &\leq \max\lprp{\epsilon, \finret},
  \end{align*}
with probability at least $1 - \delta$.

This concludes the proof of the theorem.

\qed
\end{document}